\documentclass[runningheads]{llncs}
\usepackage[T2A]{fontenc}
\usepackage[utf8]{inputenc}
\usepackage{url}
\usepackage{multirow}
\usepackage{amssymb}
\usepackage{amsxtra}
\usepackage{makeidx}
\usepackage{amsfonts}
\usepackage{algorithm}
\usepackage{algorithmic}
\usepackage{float}
\usepackage{graphicx}
\usepackage[font=small,labelfont=bf]{caption}
\usepackage{comment}
\usepackage{color}
\usepackage{hyperref}
\usepackage{nicefrac}
\usepackage{subcaption}

\begin{document}
	\title{New Version of Mirror Prox for Variational Inequalities with Adaptation to Inexactness \thanks{The research of F. Stonyakin in Sections 2, 3 and numerical experiments in Table 1 were supported by Russian Science Foundation (project 18-71-00048). The research of E. Vorontsova was supported by Russian Foundation for Basic Research (project number 18-29-03071 mk and project number 18-31-20005 mol-a-ved).}}
	\titlerunning{Mirror Prox for VI with Adaptation to Inexactness}
	%
	\author{Fedor S. Stonyakin \inst{1,2} \orcidID{0000-0002-9250-4438}  \and 
	Evgeniya A. Vorontsova\inst{3, 4}\orcidID{0000-0003-2173-6503} \and
		Mohammad S. Alkousa	\inst{2} \orcidID{0000-0001-5470-0182} }
	\authorrunning{F. Stonyakin et al.}
	%
	\institute{V.\,I.\,Vernadsky Crimean Federal University, Simferopol, Russia\\
		\email{fedyor@mail.ru}
	\and
	Moscow Institute of Physics and Technology, Moscow, Russia\\
		\email{mohammad.alkousa@phystech.edu}
\and
Grenoble Alpes University, Grenoble, France
\and
Far Eastern Federal University, Vladivostok, Russia \\
\email{vorontsovaea@gmail.com}
		}
	\maketitle
	\begin{abstract}
		Some adaptive analogue of the Mirror Prox method for variational inequalities is proposed. In this work we consider the adaptation not only to
		the value of the Lipschitz constant,
		but also to the magnitude of the oracle error. This approach, in particular, allows us to prove a complexity near $O\left(\frac{1}{\varepsilon}\log_2\frac{1}{\varepsilon}\right)$ for variational inequalities for a special class of 
		monotone bounded operators. This estimate is optimal for variational inequalities with monotone Lipschitz-continuous 
		operators. However, there exists some error, which may be insignificant. The results of experiments on the comparison of the proposed approach with some known analogues are presented. Also, we discuss the results of the experiments for matrix games in the case of using non-Euclidean proximal setup.
	\end{abstract}

\section{Introduction}

Variational inequalities (VI) and saddle point problems often arise in a variety of important applications
\cite{FaccPang_2003}. For solving such problems a lot of algorithmic schemes are known (see e.g. \cite{FaccPang_2003,Antipin_2017,Chambolle,Nemirovski_2004}). 
The Mirror Prox method proposed by A.\,S.\,Nemirovski \cite{Nemirovski_2004}
is currently one of the most popular of such methods. This method goes back to the well-known extragradient method proposed by G.\,M.\,Korpelevich in \cite{Korpelevich}. At the same time, unlike the standard extragradient method, Mirror Prox allows to effectively solve problems with non-Euclidean norms, as well as with the Holder-continuous operators:
\begin{equation}\label{eq001}
\|g(x)-g(y)\|_*\leqslant L_{\nu}\|x-y\|^{\nu}\; \, \forall \, x, \, y \, \in \, Q\text{ for some }\nu\in[0;1],
\end{equation}
all notations are explained in Sect.~\ref{section:Problem_Statement} below.

Recently, a universal analogue of the A.\,S.\,Nemirovski  method~\cite{Nemirovski_2004} was proposed in \cite{UMP_MAIN_1,UMP_MAIN}. The universality is understood as an adaptive adjustment to the optimal smoothness level~$\nu$
in \eqref{eq001}, as well as the constant value $L_\nu >0$. Note that universal gradient method for unconstrained convex optimization problems was proposed by Yu.E.~Nesterov~\cite{Nesterov_2015} (see also Sect.~5 of the textbook~\cite{Gasn_2017}). And it is possible to observe the convergence rate of the proposed method, which is typical for the smooth case $\nu=1$ (Lipschitz-continuous operators), for some problems with bounded operators ($L_0 < + \infty$ and $L_\nu = + \infty$ for all $\nu > 0$).

This paper is devoted to the modification of Mirror Prox method \cite{UMP_MAIN_1,UMP_MAIN} for the following analogue of the Lipschitz condition for the operator $g$ with constant $L>0$ 
\begin{equation}\label{eqv_gen_smooth}
\langle g(y)-g(x), \, y-z\rangle\leqslant LV(y,x) + LV(y,z)+\delta\|y-z\|\;\;\;\forall x,y,z\in Q,
\end{equation}
where $\delta > 0$ is a fixed
value and $V(y,x)$ is the Bregman divergence (see Sect.~\ref{section:Problem_Statement} below).

We propose an analogue of аdaptive Mirror Prox method \cite{UMP_MAIN_1,UMP_MAIN}. At the same time, we consider adaptive tuning both for the value of the parameters $L$ and $\delta$. 
One of the features of the proposed method which are important for applications
is the possibility for 
the value of $\delta$ 
to reflect the inexactness of operator $g$. In addition, the value of $\delta$ can indicate the degree of discontinuity of the operator $g$. Adaptive tuning to its value can approximate the convergence rate for variational inequalities with bounded operators ($\nu = 0$) to the convergence rate for variational inequalities with Lipshitz-continuous operators ($\nu = 1$). Effects of this approach can be observed for the universal method but without a theoretical justification for the convergence rate $O\left(\frac{1}{\varepsilon}\right)$ for non-smooth operators \cite{UMP_MAIN_1}. This means that the proposed approach in this paper is an alternative to the universal method.

The contribution of the paper can be summarized as follows:

- An analogue of the Mirror Prox method for variational inequalities with a monotone Lipschitz-continuous operator, which allows for adaptive tuning to the value of the Lipschitz constant $L$, as well as the limit value of the error $ \delta $ of the specifying operator $g$, is proposed.

- The applicability of the proposed method to a certain class of variational inequalities with bounded operators ($\nu = 0$) is discussed. The rate of convergence $O\left(\frac{1}{\varepsilon}\right)$ of this method is proved with some finite error associated with the non-smoothness of the operator. Thus, some alternative to the universal method has been proposed, but with a clearer theoretical rationale for acceleration.


- The results of numerical experiments for finding the equilibrium in a bilinear matrix game (or VI with Lipschitz-continuous operator) with a bounded error in the definition of the operator are given. A comparison of the quality of the calculated solution is given depending on the number of iterations for the adaptive Mirror Prox method from~\cite{UMP_MAIN}
and the method proposed in this paper with an adaptive setting for the magnitude of the error.

- The results of numerical experiments for a variational inequality with a bounded ($\nu = 0$) operator (related to the Fermat-Torricelli-Steiner problem) are 
presented. These results show that the method due to the proposed adaptation of the non-smoothness error can converge much faster than the optimal lower estimate $O\left(\frac{1}{\varepsilon^2}\right)$ for the corresponding class of problems.

\section{Problem Statement and Some Examples}\label{section:Problem_Statement}
Let $(E, \, \|\cdot\|)$ be a normed finite-dimensional vector space and $E^*$ be the conjugate space of $E$ with the norm:
$$||y||_*=\max\limits_x\{\langle y, \, x\rangle, \, ||x||\leq1\},$$
where $\langle y, \, x\rangle$ is the value of the continuous linear functional $y$ at $x \in E$.

Let $Q\subset E$ be a (simple) closed convex set, $d : Q \rightarrow \mathbb{R}$ be a distance generating function (d.g.f.), which is continuously differentiable and $1$-strongly convex with respect to the norm $\|\cdot\|$
and assume that $\min\limits_{x\in Q} d(x) = d(0).$

For all $x, \, y\in Q \subset E$ we consider the corresponding Bregman divergence
\begin{equation}\label{Bregman_diver}
    V(x , y) = d(x) - d(y) - \langle \nabla d(y), \, x - y \rangle.
\end{equation}

Let $g: Q \to E^*$ be a continuous operator. We  consider the problem of finding a solution to a variational inequality of the form
\begin{equation}\label{ad_meth_eq1}
\langle g(x_*), \, x_*-x \rangle\leqslant0\;\;\forall \, x \in Q.
\end{equation}
Under the assumption of the monotony of the operator $g$, i.e. 
$$
    \langle g(x)-g(y), \, x-y \rangle \geq 0 \;\; \forall \, x, \, y \in Q,
$$
the inequality \eqref{ad_meth_eq1} is equivalent to the following weak variational inequality
\begin{equation}\label{ad_meth_eq2}
\langle g(x), \, x_*-x \rangle     \leqslant0\;\;\forall \, x\in Q.
\end{equation}
Assume that the operator $g$ satisfies the condition \eqref{eqv_gen_smooth}. In this section, we show some examples of problems for which a condition of the form \eqref{eqv_gen_smooth} naturally arises. First of all, this is due to the inexactness of the oracle for the operator of a variational inequality. But also the value of $\delta\|y-z\|$ can describe the degree of discontinuity of the operator $g$ (i.e. using considered approach, one can propose an approach to the solution of some VI's with bounded operators).

\begin{example}
Let $g:Q\rightarrow\mathbb{R}^n$ be a Lipschitz-continuous operator with constant $L>0$, i.e.
$$\|g(x)-g(y)\|_*\leqslant L\|x-y\|\;\;\forall \, x, \, y\in Q.$$
However, suppose that the exact value of the operator $g$ is not available, and only an approximate value of $g(x)$, i.e. $\tilde{g}(x)$, is known:
$$
    \|\tilde{g}(x)-g(x)\|_*\leqslant\frac{\delta}{2} \quad \forall \, x\in Q.
$$

Then for each $x, \, y, \, z\in Q$ we have:
$$|\langle\tilde{g}(y)-\tilde{g}(x), \, y-z\rangle-\langle g(y)-g(x), \, y-z\rangle|=|\langle\tilde{g}(y)-g(y), \, y-z\rangle+\langle g(x)-\tilde{g}(x), \, y-z\rangle|\leqslant$$
$$\leqslant\|\tilde{g}(y)-g(y)\|_*\cdot\|y-z\|+\|\tilde{g}(x)-g(x)\|_*  \cdot\|y-z\|\leqslant  \left( \frac{\delta}{2}+\frac{\delta}{2} \right) \|y-z\|=\delta\|y-z\|.
$$
Therefore,
$$\langle\tilde{g}(y)-\tilde{g}(x), \, y-z\rangle\leqslant\langle g(y)-g(x), \, y-z\rangle+\delta\|y-z\|\leqslant\|g(y)-g(x)\|_*\cdot\|y-z\|+\delta\|y-z\| \leqslant$$
$$\leqslant L \|y-x\| \cdot \|y-z\| + \delta \|y-z\| \leqslant \frac{L}{2} \|y-x\|^2 + \frac{L}{2}\|y-z\|^2 + \delta \|y-z\| \leqslant $$
$$\leqslant LV(y,x)+LV(y,z)+\delta.$$
\end{example}

\begin{example}
Note that the term $\delta\|y-z\|$  in \eqref{eqv_gen_smooth} can describe   non-smoothness for the operator $g$ along any fixed vector segment $\{ty+(1-t)x\}_{0\leq t \leq 1}$. 
In general (if you combine all possible vector segments), on the domain of non-smoothness points there can be an infinite number.

For example, assume that for some subset $Q_0 \subset Q$ the function $f$ is differentiable at all points of $Q\setminus Q_0$ and that for an arbitrary $x\in Q_0$ there exists a finite subdifferential $\partial f(x)$ in the sense of convex analysis.

For fixed $x, y \in Q $ with $t\in[0;1]$ we denote $y_t:=(1-t)x+ty$. 
	
\begin{definition} (\cite{St_new})
Fix $\delta>0$ and $L>0$. We say that the convex function $f:Q\rightarrow\mathbb{R}\;(Q\subset\mathbb{R}^n)$ has $(\delta,L)$-Lipschitz subgradient ($f\in C_{L,\delta}^{1,1}(Q)$), if:
		\begin{itemize}
			\item[(i)] for arbitrary $x,y\in Q$ the function $f$ is differentiable at all points of the set $\{y_t\}_{0\leqslant t\leqslant1}$, with the exception of the sequence (possibly finite)
			\begin{equation}\label{ad_meth_eq10}
			\{y_{t_k}\}_{k=1}^{\infty}:\;t_1 < t_2<t_3 < \ldots \, \text{ and} \, \lim_{k\rightarrow\infty}t_k = 1;
			\end{equation}
			\item[(ii)]for a sequence of points from \eqref{ad_meth_eq10} there exist finite subdifferentials $\{\partial f(y_{t_k})\}_{k=1}^{\infty}$ and
			\begin{equation}\label{ad_meth_eq11}
			diam\;\partial f(y_{t_k})=:\delta_k>0,\text{ where }\sum_{k=1}^{+\infty}\delta_k=:\delta<+\infty,
			\end{equation}
			and \, $diam\;\partial f(x)=\max\{\|y-z\|_*\,\mid\,y,z\in\partial f(x)\};$
			\item[(iii)] if for $x,y\in Q$ the function $f$ is differentiable at each point $y_t$, $t \in (0; 1)$, then the following inequality holds:
			\begin{equation}\label{ad_meth_eq12}
			\min_{\substack{\hat{\partial}f(x)\in\partial f(x),\\\hat{\partial}f(y)\in\partial f(y)}}\|\hat{\partial}f(x)-\hat{\partial}f(y)\|_*\leqslant L\|x-y\|.
			\end{equation}
			\end{itemize}
	\end{definition}
	
Indeed, the property of $(\delta,L)$-Lipschitzness for each subgradient $g(x) = \hat{\partial}f(x)$ means that
\begin{equation}\label{ad_meth_eq4}
\|g(y)-g(x)\|_*\leqslant L\|y-x\|+\delta.
\end{equation}
	
To prove \eqref{ad_meth_eq4}, it suffices to split the segment $\{y_t\}_{0 \leqslant t \leqslant1}$ into the intervals of smoothness and take into account the boundedness of diameters of the subdifferentials at non-smoothness points of $f$.
	
The inequality \eqref{ad_meth_eq4} means that
$$\langle g(y) - g(x), \, y - z \rangle  \leqslant\|g(y)-g(x)\|_*\cdot\|y-z\|\leqslant L\|y-x\|\cdot\|y-z\|+\delta\|y-z\|\leqslant$$
$$\leqslant\frac{L}{2} (\|y-x\|^2+\|y-z\|^2) + \delta \|y-z\| \leqslant L V(y,x) + L V(z,y)+\delta\|y-z\|.$$
\end{example}

Let us give a concrete example of a non-smooth functional with a ($\delta$, L)-Lipschitz subgradient with an arbitrarily large Lipschitz constant.
\begin{example}
We fix some $k>0$, the value $\delta>0$ and consider the piecewise linear function $f:[0;1]\rightarrow\mathbb{R}$ (here $Q=[0;1]\subset\mathbb{R}$) defined as follows
\begin{equation}\label{ad_meth_eq7}
    f(x) = 
	\begin{cases}
	kx & ;0\leqslant x\leqslant\frac{1}{2},\\
	\left(k+\sum_{i=1}^{n}\frac{\delta}{2^i} \right) x - \sum_{i=1}^n\frac{\delta}{2^i} \left(1-\frac{1}{2^i}\right) & ;1-\frac{1}{2^n}<x\leqslant1-\frac{1}{2^{n+1}}, \\
	\lim_{x\rightarrow+1}f(x) & ;x = 1.
	\end{cases}
	\end{equation}
	
	In this case,  $Q_0 = \{1-\frac{1}{2^n}\}_{n=1}^{\infty}$,  $\partial f(q_n)=\left[k+\sum\limits_{i=1}^{n-1}\frac{\delta}{2^i};k+\sum\limits_{i=1}^n\frac{\delta}{2^i}\right]$ with  $n>1$, $\partial f(q_1)=\left[k;k+\frac{\delta}{2}\right]$ (here $q_n=1-\frac{1}{2^n}$ with $n=1,2,3,\ldots$). It is clear that $\partial f(q_n)=\frac{\delta}{2^n}$, which is true for the entered value $\delta>0$. Moreover, on the intervals $(0;q_1),\; (q_n;q_{n+1})$  the function $f$ has a Lipschitz-continuous gradient with the constant $L=0$. Therefore, for the function $f$ from \eqref{ad_meth_eq7}, we find that $f\in C_{0,\delta}^{1,1}(Q)$.
\end{example}

Any functional with a finite set of non-smooth points along an arbitrary segment will satisfy the proposed Lipschitz condition for the subgradient. Obviously, this condition holds for each objective function with finite points of non-smoothnes on each vector segment $[x;y]$. Thus, it is possible to apply this technique to problems of minimization for sum distances to several balls in Hilbert spaces \cite{Mordukh}. Such an objective function, obviously, will not be differentiable in the usual sense at the points of the boundaries of the balls of which there are infinitely many. Note that among points of each vector segment $[x;y]$ such an objective function have finite points of non-smoothness. However, the considered Lipschitz condition for a special choice of subgradient holds for some functions with infinitely many points of non-smoothness (e.g. for maximum of linear functions).

\section{Adaptive Method for Variational Inequalities with Adaptation to Inexactness}

In this section, we introduce a new version of the Mirror Prox method for variational inequalities (see \cite{UMP_MAIN}), which we call \textit{Mirror Prox with Adaptation to Inexactness (MPAI)}. In this version, which is listed as Algorithm \ref{algorithm:new_GMP} below, we consider the adaptation not only to the level of operator smoothness, but also to the magnitude of the oracle error, which may allow to receive complexity near $O\left(\frac{1}{\varepsilon}\right)$ for VI with bounded operators, i.e. the optimal complexity for VI with Lipshitz-continuous operators. 

We evaluate the solution quality of the problem \eqref{ad_meth_eq1}, produced by Algorithm~\ref{algorithm:new_GMP}, by using the Bregman divergence
\eqref{Bregman_diver}.

\begin{algorithm}
	\caption{Mirror Prox with Adaptation to Inexactness (MPAI).}
	\label{algorithm:new_GMP}
	{\bf Input:} $x^0=\arg\min\limits_{x\in Q}d(x),L^0,\delta^0$.
	\begin{algorithmic}[1]
		\STATE $N:=N+1;\;L^{N+1}:=\frac{L^N}{2};\;\delta^{N+1}:=\frac{\delta^N}{2}$.\\
		\STATE Calculate
		\begin{equation}\label{Stonextragrad1}
         y^{N+1}:=\arg\min\limits_{x\in Q} \{ \langle g(x^N), \, x-x^N \rangle + L^{N+1}V(x,x^N)\},
		\end{equation}
		\begin{equation}\label{Stonextragrad2}
		x^{N+1}:=\arg\min\limits_{x\in Q}  \{ \langle g(y^{N+1}), \, x-x^N \rangle  +L^{N+1}V(x,x^N) \}.
		\end{equation}
		\STATE {\bf If} 
		\begin{equation}\label{ineq:go_out_from_iteration}
		\langle g(y^{N+1})-g(x^N), \, y^{N+1}-x^{N+1} \rangle \leq  L^{N+1}V(y^{N+1},x^N) + 
		\end{equation}
		$$
		+ L^{N+1}V(x^{N+1},y^{N+1})+ \delta^{N+1} \left \|y^{N+1}-x^{N+1} \right \|,
		$$
		{\bf then} go to the next iteration (item 1).\\
		\STATE {\bf Else} increase $L^{N+1}$ and $\delta^{N+1}$ by two times and go to item 2.\\
\end{algorithmic}
\end{algorithm}

\begin{theorem}\label{ad_meth_th1}
	After $N$ iterations of Algorithm~\ref{algorithm:new_GMP}, the following estimate holds:
	$$\sum_{k=0}^{N-1}\frac{1}{L^{k+1}}\langle g(y^{k+1}), \, y^{k+1}-x \rangle \leqslant V(x,x^0)-V(x,x^N)+\sum_{k=0}^{N-1}\frac{\delta^{k+1}}{L^{k+1}} \left \|y^{k+1}-x^{k+1} \right \|.$$
\end{theorem}

\begin{proof}
One can directly check the following inequalities:
\begin{equation}\label{7}
\left\langle \nabla_x V(x,x^k)\big|_{x=x^{k+1}}, \, x-x^{k+1}\right\rangle= V(x,x^k)-V(x,x^{k+1})-V(x^{k+1},x^k),
\end{equation}
\begin{equation}\label{8}
\left\langle \nabla_x V(x,x^k)\big|_{x=y^{k+1}}, \, x-y^{k+1}\right\rangle= V(x,x^k)-V(x,y^{k+1})-V(y^{k+1},x^k).
\end{equation}
Further, for each $x\in Q$ and $k=\overline{0,\ N-1}$:
$$\left\langle \nabla_x \left(\left\langle g(x^k), \, x-x^k \right\rangle+L^{k+1}V(x,x^k)\right)\big|_{x=y^{k+1}},\ x-y^{k+1} \right\rangle \geqslant 0,$$
$$\left\langle \nabla_x \left(\left\langle g(y^{k+1}), \, x-x^k \right\rangle+L^{k+1}V(x,x^k)\right)\big|_{x=x^{k+1}},\ x-x^{k+1} \right\rangle \geqslant 0.$$

Thus,
$$\left\langle g(y^{k+1}), \, x^{k+1}-x \right\rangle \leqslant L^{k+1}V(x,x^k)- L^{k+1}V(x,x^{k+1})-L^{k+1}V(x^{k+1},x^k)$$
and
$$\left\langle g(x^{k}), \, y^{k+1}-x \right\rangle \leqslant L^{k+1}V(x,x^{k})-L^{k+1}V(x,y^{k+1})-L^{k+1}V(y^{k+1},x^k).$$

Taking into account \eqref{ineq:go_out_from_iteration}, we have for each $k=\overline{0,\ N-1}$:
$$\left\langle g(y^{k+1}), \, y^{k+1}-x \right\rangle = \left\langle g(y^{k+1}), \, x^{k+1}-x \right\rangle+ \left\langle g(x^{k}), \, y^{k+1}-x^{k+1}\right\rangle +$$
$$+ \left\langle g(y^{k+1})-g(x^k), \, y^{k+1}-x^{k+1}\right\rangle \leqslant$$
$$\leqslant L^{k+1}V(x,x^{k})-L^{k+1}V(x,x^{k+1})-L^{k+1}V(x^{k+1},x^{k})+L^{k+1}V(x^{k+1},x^{k})-$$
$$-L^{k+1}V(x^{k+1},y^{k+1})-L^{k+1}V(y^{k+1},x^{k})+L^{k+1}V(y^{k+1},x^{k})+$$
$$
+L^{k+1}V(x^{k+1},y^{k+1}) +
\delta^{k+1} \left \|y^{k+1}-x^{k+1} \right \|,
$$
i.e.
\begin{equation}\label{9}
\frac{1}{L^{k+1}} \left\langle g(y^{k+1}), \, y^{k+1}-x \right\rangle \leqslant V(x,x^k)-V(x,x^{k+1}) + 
\frac{\delta^{k+1}}{L^{k+1}}
\left \|y^{k+1}-x^{k+1} \right \|.
\end{equation}

After summing \eqref{9} by $k=\overline{0,\ N-1}$, we have
$$\sum_{k=0}^{N-1} \frac{1}{L^{k+1}} \langle g(y^{k+1}), \, y^{k+1}-x\rangle \leqslant V(x,x^0)-V(x,x^N) +\sum_{k=0}^{N-1}\frac{\delta^{k+1}}{L^{k+1}}
\left \|y^{k+1}-x^{k+1} \right \|.$$
$\blacksquare$
\end{proof}

Let us denote 
$$S_N=\sum\limits_{k = 0}^{N-1}\frac{1}{L^{k+1}}, \,
\widetilde{y} = \frac{1}{S_N}\sum\limits_{k = 0}^{N-1} \frac{y^{k+1}}{L^{k+1}}  \text{  and   } R^2=\max_{x\in Q}V(x,x^0).$$

\begin{theorem}\label{ad_meth_th2}
For monotone operator $g$	after $N$ iterations of Algorithm \ref{algorithm:new_GMP},
	the following estimate holds: 
	\begin{equation}\label{ineq:estimate_of_GMP}
	\max_{x\in Q}\langle g(x), \, \widetilde{y} - x \rangle \leqslant \frac{R^2}{S_N}+\frac{1}{S_N}\sum_{k=0}^{N-1}\frac{\delta^{k+1}}{L^{k+1}}
	\left \|y^{k+1}-x^{k+1} \right \|.
	\end{equation}
Assume that for fixed $\varepsilon$ 
\begin{equation}\label{stopping_rule}
\sum\limits_{k=0}^{N-1}\frac{1}{L^{k+1}}\geqslant\frac{R^2}{\varepsilon}.
\end{equation}
Then the following inequality holds:
	\begin{equation}\label{ineq:estimateofGMP}
	\max_{x\in Q} \langle g(x), \, \widetilde{y} - x \rangle \leqslant \varepsilon + \frac{1}{S_N}\sum_{k=0}^{N-1}\frac{\delta^{k+1}}{L^{k+1}} \left \|y^{k+1}-x^{k+1} \right \|.
\end{equation}
If $L^0\leqslant2L$, then  inequality \eqref{stopping_rule} holds at no more than $$N=\left\lceil\frac{2LR^2}{\varepsilon}\right\rceil$$ 
iterations of Algorithm \ref{algorithm:new_GMP}. 
\end{theorem}

\begin{proof}
By monotony of $g$ we have for each $k = 0, 1, ...$:
	$$ \langle g(x), \, y^{k+1}-x \rangle = \langle g(y^{k+1}), \, y^{k+1}-x \rangle + \langle g(x)-g(y^{k+1}), \, y^{k+1}-x\rangle \leqslant \left \langle g(y^{k+1}), \, y^{k+1}-x \right\rangle,$$
	so the inequality
	\begin{equation}\label{ad_meth_eq8}
	\begin{split}
	\frac{1}{S_N} \max_{x\in Q}\sum\limits_{k=0}^{N-1}\frac{1}{L^{k+1}} \left\langle g(y^{k+1}), \, y^{k+1}-x \right\rangle\leqslant\\ 
	\leqslant \frac{R^2}{S_N}+\frac{1}{S_N}\sum_{k=0}^{N-1}\frac{\delta^{k+1}}{L^{k+1}} \left \|y^{k+1}-x^{k+1} \right \| \leqslant \varepsilon + \frac{1}{S_N}\sum_{k=0}^{N-1}\frac{\delta^{k+1}}{L^{k+1}}\left \|y^{k+1}-x^{k+1} \right \|
	\end{split}
	\end{equation}
	can be replaced by
	\begin{equation}\label{ad_meth_eq9}
	\begin{split}
	    \max_{x\in Q} \left \langle g(x), \, \widetilde{y}-x \right\rangle \leqslant \frac{R^2}{S_N}+\frac{1}{S_N}\sum_{k=0}^{N-1}\frac{\delta^{k+1}}{L^{k+1}}\left \|y^{k+1}-x^{k+1} \right \|\leqslant\\ \leqslant \varepsilon + \frac{1}{S_N}\sum_{k=0}^{N-1}\frac{\delta^{k+1}}{L^{k+1}}\left \|y^{k+1}-x^{k+1} \right \|.
	\end{split}
	\end{equation}
	
	\end{proof}
	
\begin{remark}
Due to adaptive choice of parameters $L^{k+1}$ and $\delta^{k+1}$ at each iteration of Algorithm \ref{algorithm:new_GMP} the expression 
$$
\frac{R^2}{S_N}+\frac{1}{S_N}\sum_{k=0}^{N-1}\frac{\delta^{k+1}}{L^{k+1}}\left \|y^{k+1}-x^{k+1} \right \|
$$
in \eqref{ad_meth_eq9} may be small enough even in the case of $L = + \infty$ or $\delta = + \infty$ in \eqref{eqv_gen_smooth}.
\end{remark}
	
\begin{remark}
Clearly, for each $k$, we have $\delta_k \leqslant C_L\delta$ ($C_L = \max\left\{1, \frac{2L}{L^0}\right\}$) and:
$$
    \frac{1}{S_N}\sum_{k=0}^{N-1} \frac{\delta^{k+1}}{L^{k+1}}\left \|y^{k+1}-x^{k+1} \right \| \leqslant C_L\delta \max_{k = \overline{0, N-1}} \left \|y^{k+1}-x^{k+1} \right \|. 
$$
 
This means that the value associated with the error in the specifying operator~$g$ is bounded on the set~$Q$ of a finite diameter.
\end{remark}

\begin{remark}
	If $g \not\equiv 0$, then the condition $L^0 \leqslant 2 L$ can be satisfied by choosing
	$$L^0 := \frac{\|g(x) - g(y)\|_{*}}{\|x - y\|} \; \text{at} \; g(x) \neq g(y).$$
\end{remark}

\begin{remark}
	Note that the estimate of the number of iterations $ N = \left \lceil \frac{2LR^2} {\varepsilon} \right \rceil $ with accuracy to a numerical factor is optimal for variational inequalities with a Lipschitz-continuous operator \cite{Nemirovsky_compl}.
	Note that the evaluation of the inexactness of the value of the operator, as we see from the previous remark, is bounded and does not accumulate.
	
	Note that similarly Remark 4 in \cite{St_Num_Anal} the total number of attempts to solve \eqref{Stonextragrad1} and \eqref{Stonextragrad2} is bounded by $4N + \max\left\{\log_2\frac{2L}{L^0},~ \log_2\frac{2\delta}{\delta^0}\right\}$.

\end{remark}

\begin{remark}\label{RemNonSmooth}  

It was pointed earlier that $\delta $ can describe the degree of non-smoothness of the operator $g$. Let us show how Algorithm \ref{algorithm:new_GMP} can be slightly modified in order to provably 
obtain an approximate complexity estimate
$$
    O\left(\frac{1}{\varepsilon^2}\log_{2}\frac{1}{\varepsilon}\right)
$$
to achieve the quality of the solution $\widetilde{y}$:
$$
    \displaystyle\max_{x\in Q}\langle g(x), \, \widetilde{y}-x \rangle\leqslant \varepsilon
$$
in the case of a bounded operator $g$ (generally speaking, non-smooth).

Suppose that at each iteration $L\leqslant L_{k+1}\leqslant 2L$ (this can always be achieved with a constant number of calculations in item 2 of the listing of Algorithm \ref{algorithm:new_GMP}).

We propose such a procedure for some positive integer $p$: repeat the operations item 2 of Algorithm \ref{algorithm:new_GMP}, $p$ times, each time
\begin{equation}\label{B1}
L_{k+1}:=2 \cdot L_{k+1} \;\; \text{with the same}\;\; \delta_{k+1}\leqslant 2\delta
\end{equation}
(for each iteration of the algorithm \ref{algorithm:new_GMP}, $k=0, \, 1, \, 2, \, \ldots, \, N-1$). A procedure of type \eqref{B1} will be stopped if one of the following two inequalities holds:
\begin{equation}\label{B2}
\delta_{k+1} \left \|y^{k+1}-x^{k+1} \right \|\leqslant\frac{\varepsilon}{2},
\end{equation}
or
\begin{equation}\label{B3}
\left\langle g(y^{k+1})-g(x^{k}), \, y^{k+1}-x^{k+1}\right\rangle\leqslant2^{p-1}L\left( \left \|y^{k+1}-x^{k} \right \|^{2}+ \left \|y^{k+1}-x^{k+1} \right \|^{2}\right).
\end{equation}

Note that the  procedure \eqref{B1} involves updating $x^{k+1}$ and $y^{k+1}$ during repetitions. Let us estimate $p$
(the number of times required for executing \eqref{B2} or \eqref{B3} to repeat a procedure of type \eqref{B1}).

It is clear that  $\forall\,x^{k}, x^{k+1}, y^{k+1}\in Q$ it is true that
$$
\left\langle g(y^{k+1})-g(x^{k}), \, y^{k+1}-x^{k+1}\right\rangle\leqslant\frac{L}{2} \left \|y^{k+1}-x^{k} \right \|^{2}+\frac{L}{2} \left \|y^{k+1}-x^{k+1} \right \|^{2}+\delta \left \|y^{k+1}-x^{k+1} \right \|.
$$
Moreover, $\delta_{k+1}\leqslant 2\delta$. If \eqref{B2} is not true, then $\left \|y^{k+1}-x^{k+1} \right \|>\displaystyle\frac{\varepsilon}{4\delta}$ and the inequality \eqref{B3} is obviously satisfied at
\begin{equation}\label{B4}
2^{p}>1+\frac{16\delta^{2}}{L\varepsilon},
\end{equation}
since in that case
$$
\frac{2^{p}-1}{2}L \left \|y^{k+1}-x^{k+1} \right \|^{2}>\delta \left \|y^{k+1}-x^{k+1} \right \|.
$$

So, after repeating the $p$ procedures of type \eqref{B1} at each of the $N$ iterations of the method, the following inequality holds:
$$
    \max_{x\in Q}\langle g(x), \, \widetilde{y}-x \rangle\leqslant\frac{R^{2}}{S_{N}} +\frac{\varepsilon}{2}.
$$
Further,
$$
    S_{N}=\sum_{k=0}^{N-1}\frac{1}{L_{k+1}}\geqslant\frac{N}{2^{p+1}L}.
$$

Therefore, $\displaystyle\frac{R^{2}}{S_{N}}\leqslant\frac{2^{p+1}L R^{2}}{N}\leqslant \frac{\varepsilon}{2}$ at $N\geqslant\displaystyle\frac{2^{p+2}L R^{2}}{\varepsilon}$, whence, taking into account \eqref{B4}
\begin{equation}\label{B5}
N\geqslant\frac{4L R^{2}}{\varepsilon}+\frac{64\delta^{2}R^{2}}{\varepsilon^{2}}.
\end{equation}

Generally speaking, we need $O\left(\log\displaystyle\frac{1}{\varepsilon}\right)$ additional steps of item 2 of listing 1 of Algorithm \ref{algorithm:new_GMP} at each iteration. So, the final estimate of complexity to achieve the quality of $\displaystyle\max_{x\in Q} \langle g(x), \, \widetilde{y}-x \rangle\leqslant \varepsilon$  will be $O\left(\displaystyle\frac{1}{\varepsilon^{2}}\log_{2}\displaystyle\frac{1}{\varepsilon}\right)$. It is well-known that this estimate is optimal up to a logarithmic factor.

However, for small enough $\delta$ in \eqref{B5} we have complexity near $O\left(\displaystyle\frac{1}{\varepsilon}\log_{2}\displaystyle\frac{1}{\varepsilon}\right)$.
\end{remark}

\section{Numerical Experiments for Non-smooth Problem: Variational Inequality for Some Analogue of Fermat-Torricelli-Steiner Problem} \label{section:numerical1}

In this section, to show the advantages of the proposed Algorithm \ref{algorithm:new_GMP}, we consider some numerical experiments for the saddle point problem (and the corresponding VI), which corresponds to the convex programming problem for some analogues of the Fermat-Torricelli-Steiner problem with functional constraints. Note that the objective functions are non-smooth and the corresponding operators of the variational inequality of the problem under consideration are bounded ($\nu = 0$). However, experimentally, due to adaptation, we can observe an estimate of the complexity inherent in the case of Lipshitz-continuous operators of VI.

All experiments in this section were implemented in Python 3.4, on a computer equipped with Intel(R) Core(TM) i7-8550U CPU @ 1.80GHz, 1992 Mhz, 4 Core(s), 8 Logical Processor(s). RAM of the computer is 8GB.

 For a given set of $N$ points $\{A_k=(a_{1k},a_{2k},\ldots,a_{nk}); \, k=\overline{1,N}\}$, that represent the centers of the balls $\omega_k$ with radii $r_k$, in the $n$-dimensional Euclidean space  $\mathbb{R}^n$, we need to find such a point $X=(x_1,x_2,\ldots,x_n)$ of the objective function \cite{Mordukh}
\begin{equation}\label{object_circles}
f(x):=\sum\limits_{k=1}^{N}d(X,A_k),
\end{equation}
where
$$
d(X,A_k)=
\begin{cases}
XA_k-r_k,&\mbox{if } |XA_k|\geqslant r_k\\
0,&\mbox{otherwise},
\end{cases}
$$
would take the minimal value on the set $Q$, which is given by several functional constraints:

\begin{equation}\label{eq_2}
\begin{split}
\varphi_1(x)=\alpha_{11}x_1^2+\alpha_{12}x_2^2+\ldots+\alpha_{1n}x_n^2-1,\\
\varphi_2(x)=\alpha_{21}x_1^2+\alpha_{22}x_2^2+\ldots+\alpha_{2n}x_n^2-1,\\
\ldots\\
\varphi_m(x)=\alpha_{m1}x_1^2+\alpha_{m2}x_2^2+\ldots+\alpha_{mn}x_n^2-1,
\end{split}
\end{equation}	

where the coefficients $\alpha_{11},\alpha_{12},\ldots,\alpha_{mn}$ are represented by the matrix
$$
\begin{pmatrix}
\alpha_{11} & \alpha_{13} & \dots & \alpha_{1n} \\
\alpha_{21} & \alpha_{23} & \dots & \alpha_{2n} \\
\hdotsfor{4} \\
\alpha_{m1} & \alpha_{m3} & \dots & \alpha_{mn}
\end{pmatrix},
$$
in which one element of each row is an integer belonging to the interval $(1;10)$, and the remaining elements of the row are equal to 1.

To solve such a problem, we can consider a saddle point problem $\min\limits_{x}\max\limits_{\lambda} L(x, \lambda)$, where
 $$L(x,\lambda)=f(x)+\sum\limits_{p=1}^m\lambda_p\varphi_p(x),\;\overrightarrow{\lambda}=(\lambda_1,\lambda_2,\ldots,\lambda_m).$$
Consider the corresponding variational inequality:
$$
\langle G(x_*, \, \overrightarrow{\lambda}_*), \, (x_*,\overrightarrow{\lambda}_*)-(x,\overrightarrow{\lambda})\rangle\leqslant0\;\;\forall(x,\overrightarrow{\lambda})\in B\subset\mathbb{R}^{n+m},
$$
where
$$B=\left\{(x,\overrightarrow{\lambda})\,|\,\sum\limits_{k=1}^nx_k^2+\sum\limits_{p=1}^m\lambda_p^2\leqslant1\right\},$$
$$
G(x,\lambda)=
\begin{pmatrix}
\nabla f(x)+\sum\limits_{p=1}^m\lambda_p\nabla\varphi_p(x), \\
-\varphi_1(x),-\varphi_2(x),\ldots,-\varphi_m(x)
\end{pmatrix}.
$$
We give an example for $n=100$, $m=20$, $N=5$, initial approximation 
$$(x^0, \lambda^0) = \left(\frac{1}{\sqrt{m+n}},\frac{1}{\sqrt{m+n}}\ldots,\frac{1}{\sqrt{m+n}}\right) \in \mathbb{R}^{n+m},$$
and $\delta_0 = \frac{1}{20}$. The coordinates of the points $A_k$ are chosen in such a way that $\left \|A_k \right \| \in [1; \, 2]$. We choose the standard Euclidean proximal setup as a prox-function.

Note that the centers of the balls were chosen with the norm in the interval $[1;2]$, and the radii of the balls are 1. Therefore, in a single ball with the center at zero there will be points of the boundary of the balls in which the objective function \eqref{object_circles} and the operator of the corresponding variational inequality will be bounded. At the same time, it can be shown that the diameter of the subdifferential at such points will be at least 1. It means that theoretically $\delta$ can be at least 1. However, experimentally, we can see significantly better solution quality (see the column "General estimate" in Table \ref{table:2}).

The results of the work of Algorithm \ref{algorithm:new_GMP}, for objective function \eqref{object_circles} are represented in Table \ref{table:2} below.

\begin{table}[H]
	\centering
	\caption{The results of Algorithm \ref{algorithm:new_GMP} for  objective function \eqref{object_circles}. }
	\label{table:2}
	\begin{tabular}{|c|c|c|}
		\hline
		Iterations & General estimate &  Time (sec.) \\ \hline
        17  &  0.1051  &  0.264 \\ \hline
        19  &  0.0527  &  0.291 \\ \hline
        21  &  0.0266  &  0.315 \\ \hline
        22  &  0.0212  &  0.342 \\ \hline
        23  &  0.0177  &  0.354 \\ \hline
        24  &  0.0133  &  0.364 \\ \hline
        25  &  0.0106  &  0.380 \\ \hline
        26  &  0.0082  &  0.427 \\ \hline
        27  &  0.0063  &  0.467 \\ \hline
        28  &  0.0048  &  0.423 \\ \hline
        29  &  0.0044  &  0.443 \\ \hline
	\end{tabular}
\end{table}

It is known (\cite{Nemirovsky_compl,NemYud_1979}) that for variational inequalities with bounded operators, the theoretical estimate of the complexity (the convergence rate) $O\left(\frac{1}{\varepsilon^2}\right)$ is theoretically optimal. However, experimentally we see from Table \ref{table:2} that, for example, an accuracy of $0.1051$ is achieved in $17$ iterations, and a 10-fold greater accuracy of $0.0106$ is achieved in $25$ iterations and approximately in the same time. If the method worked without adaptation and strictly according to optimal lower bounds for the specified class of problems, then the increase in costs could be approximately $100$ times. Thus, due to the adaptability of the proposed method, we observe a convergence rate close to $O\left(\ln\left(\frac{1}{\varepsilon}\right)\right)$.

Now for a given set of $N$ points $\{A_k=(a_{1k},a_{2k},\ldots,a_{nk}); \, k=\overline{1,N}\}$ in $n$-dimensional Euclidean space $\mathbb{R}^n$ we need to find such a point $x=(x_1,x_2,\ldots,x_n)$, that the objective function
\begin{equation}\label{objective_points}
    f(x):=\sum\limits_{k=1}^N\sqrt{(x_1-a_{1k})^2+\ldots+(x_n-a_{nk})^2} = \sum\limits_{k=1}^N \left \|x - A_k \right \|_2
\end{equation}
would take the minimal value on the set $Q$, which is defined by the previous constraints \eqref{eq_2}. The coordinates of the points $A_k$ for $k=\overline{1,N}$, are chosen as the rows of the matrix $A \in \mathbb{R}^{N\times n}$. The entries of the matrix $A$ are random integers in the closed interval $[-10;10]$, which are drawn from  the discrete uniform distribution. 

The results of the work of Algorithm \ref{algorithm:new_GMP}, for objective function \eqref{objective_points} and for some different values of $n, m$ and $N$, are presented in Table \ref{results:Fermat_Torricelli_points_quadratic} below. These results demonstrate the number of iterations produced by Algorithm \ref{algorithm:new_GMP} to reach the solution of the problem, the quality of the solution "General estimate", which is the right side of inequality \eqref{ineq:estimate_of_GMP}, and the running time of the algorithm in seconds.
\begin{table}[H]
\centering
	\caption{The results of Algorithm \ref{algorithm:new_GMP} for  objective function \eqref{objective_points}. }
	\label{results:Fermat_Torricelli_points_quadratic}
\begin{tabular}{|c|c|c|c|c|c|}
\hline
  \multicolumn{3}{|c|}{ $n=600, m=400, N=25$} & \multicolumn{3}{c|}{ $n=1000, m=500, N=50$} \\ \hline 
Iteration &General estimate &Time (sec.) & Iteration &General estimate &Time (sec.)       \\ \hline
 22  &0.122   &67.955  &19    &0.1343  &252.151  \\ \hline
 23  &0.061   &70.587  &20    &0.0672  &252.673  \\ \hline
 24  &0.0305  &75.107  &21    &0.0336  &266.636  \\ \hline
 25  &0.0153  &72.917  &22    &0.0168  &279.883  \\ \hline
 26  &0.0076  &76.686  &23    &0.0084  &293.866  \\ \hline
\end{tabular}
\end{table}

As we see from Table~\ref{results:Fermat_Torricelli_points_quadratic} we also observe a convergence rate close to $O\left(\ln\left(\frac{1}{\varepsilon}\right)\right)$, due to the adaptability of the proposed method.

\begin{remark}\label{numerical_results_points_in_unit_ball}
Now we take all previous parameters but with points $A_k (k=\overline{1,N})$  in the unit ball. The results of Algorithm \ref{algorithm:new_GMP}, for objective function \eqref{objective_points} and for some different values of $n$, $m$ and  $N$, with constraints \eqref{eq_2}, are presented in Table \ref{results_points_in_unit_ball} below.
    
In this case, since the points $A_k$ are chosen in the unit ball, the operator of the variational inequality is bounded. The results of experiments in Table \ref{results_points_in_unit_ball} show that the rate of convergence of the proposed method is close to $O\left(\frac{1}{\varepsilon}\right)$, which is significantly better than the optimal one, which is $O\left(\frac{1}{\varepsilon^2}\right)$, for non-smooth convex optimization problems and bounded variational inequalities \cite{NemYud_1979}.

\begin{table}[H]
	\centering
	\caption{The results of Algorithm \ref{algorithm:new_GMP} for objective function \eqref{objective_points}, with points $A_k$ in the unit ball.}
	\label{results_points_in_unit_ball}
	\begin{tabular}{|c|c|c|c|c|c|}
		\hline
		\multicolumn{3}{|c|}{$n=100, m=50, N=25 $} & \multicolumn{3}{c|}{$n=200, m=100, N=50 $} \\ \hline
	Iteration & General estimate & Time (sec.)& Iteration &General estimate &Time (sec.) \\ \hline
	318	 &0.2539  &35.805  &684  &0.2522 &441.744 \\ \hline
	468	 &0.1702  &52.809  &1016 &0.1688 &645.082 \\ \hline
	618	 &0.1279  &71.484  &1350 &0.1267 &857.185 \\ \hline
	768	 &0.1026  &87.851  &1682 &0.1015 &1049.885 \\ \hline
	918	 &0.0857  &103.686 &2016 &0.0846 &1305.006 \\ \hline
	1068 &0.0736  &123.056 &2349 &0.0727 &1534.333 \\ \hline
	1218 &0.0645  &141.044 &2683 &0.0637 &1753.489 \\ \hline
	1368 &0.0575  &153.877 &3014 &0.0567 &2026.402 \\ \hline
	1518 &0.0518  &173.538 &3348 &0.0511 &2210.362   \\ \hline
	1668 &0.0472  &190.714 &3680 &0.0465 &2301.327  \\ \hline
	1818 &0.0434  &208.966 &4014 &0.0427 &2611.942  \\ \hline
	1968 &0.0401  &228.198 &4346 &0.0394 &2970.324  \\ \hline
	2118 &0.0373  &243.970 &4678 &0.0367 &3153.905 \\ \hline
	2268 &0.0349  &253.112 &5012 &0.0343 &3387.451  \\ \hline
	2426 &0.0323  &266.583 &5346 &0.0322 &3619.831 \\ \hline
	\end{tabular}
\end{table}
\end{remark}

\section{Numerical Experiments for Matrix Games with Inexactness}
\label{subsec:games}
We continue our experiments with computing a Nash equilibrium of a matrix game. For that
purpose one should solve the saddle point problem
\begin{equation}
\label{eq_sad_prob}
\min_{x \, \in \, \Delta_n} \max_{y \, \in \, \Delta_m} x^T A y,
\end{equation}
where $x = (x_1, \, x_2, \, \ldots, \, x_n) \, \in \, \mathbb{R}^n$, $y = (y_1, \, y_2, \, \ldots, \, y_n) \, \in \, \mathbb{R}^m$, $\Delta_n$
is a standard simplex in $\mathbb{R}^n$,
i.e. $\Delta_n = \{x \, \in \, \mathbb{R}^n \, | \, x \, \ge \, 0, \, \, \sum\limits_{i = 1}^n x_i = 1\}$,
$\Delta_m$ is a standard simplex in $\mathbb{R}^m$,
$A$ is the payoff matrix for the $y$ player.
In all experiments we use payoff distributions centered at zero. 
Consider the following operator
\begin{equation}
\label{eq_gu}
g(u) = \left ( 
\begin{array}{c}
\nabla_x (x^T A y) \\
- \nabla_y (x^T A y)
\end{array}
\right ) =
\left ( 
\begin{array}{c}
 A^T y \\
-  A x
\end{array}
\right ),  \, u = (x, y) \, \in \, Q \equiv \Delta_n \times \Delta_m.
\end{equation}
The operator $g(u)$ from \eqref{eq_gu} is monotone on $Q$, and with this operator the VI~\eqref{ad_meth_eq2} has the
same solution as the saddle point problem~\eqref{eq_sad_prob}. 
So, Mirror Prox methods could be used for solving it.

In all experiments with matrix games we use
the entropy prox-function $d(x) = \sum\limits_{i = 1}^n x_i \ln x_i$ in Bregman divergence~\eqref{Bregman_diver}. Entropy
prox for matrix games on simplex is the best option (see \cite{Nest_2005}).

\begin{figure}[htb]
\centering
\begin{subfigure}{.5\textwidth}
  \centering
  \includegraphics[width=1.1\linewidth]{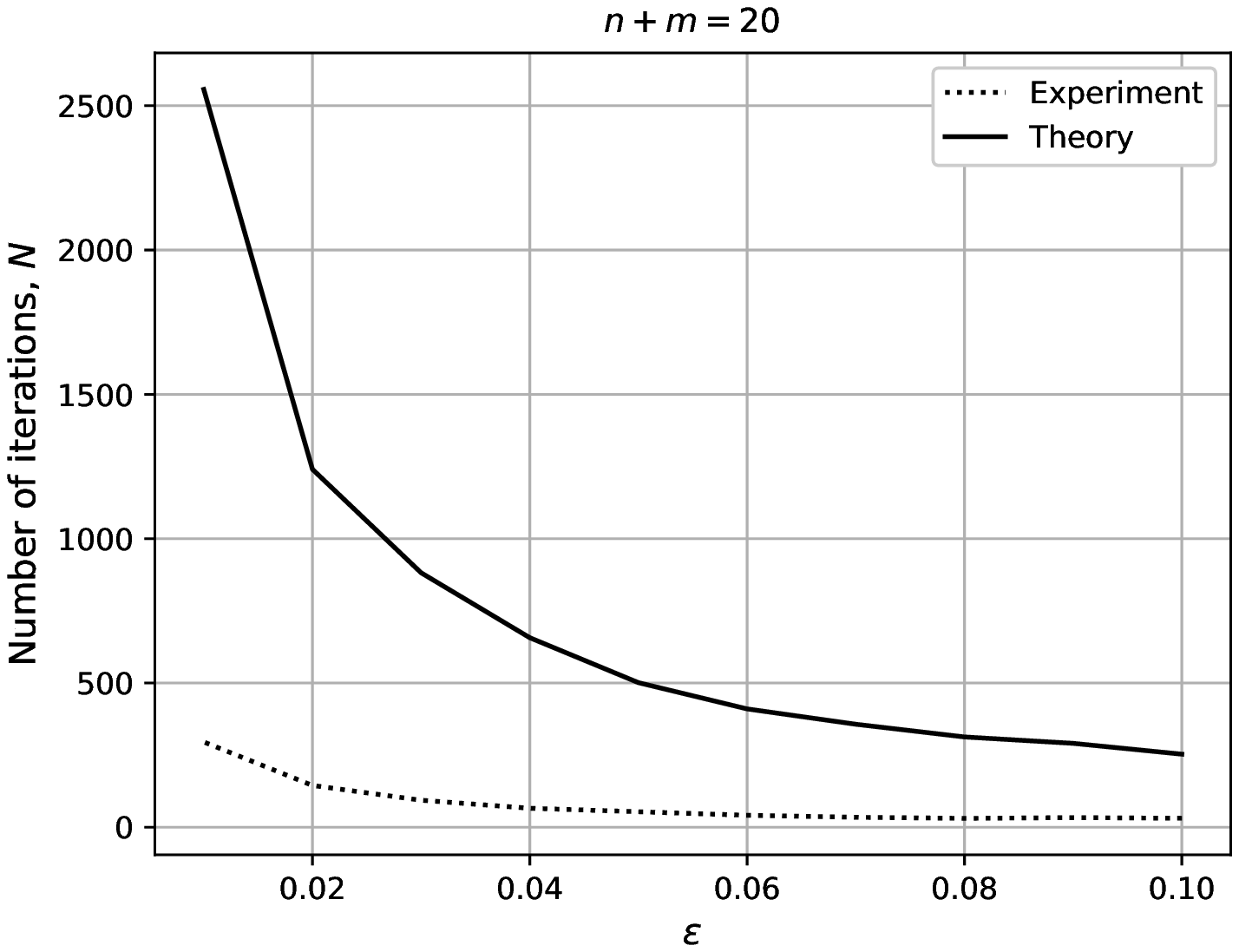}
  \caption{}
  \label{fig:sub1}
\end{subfigure}%
\begin{subfigure}{.5\textwidth}
  \centering
  \includegraphics[width=1.1\linewidth]{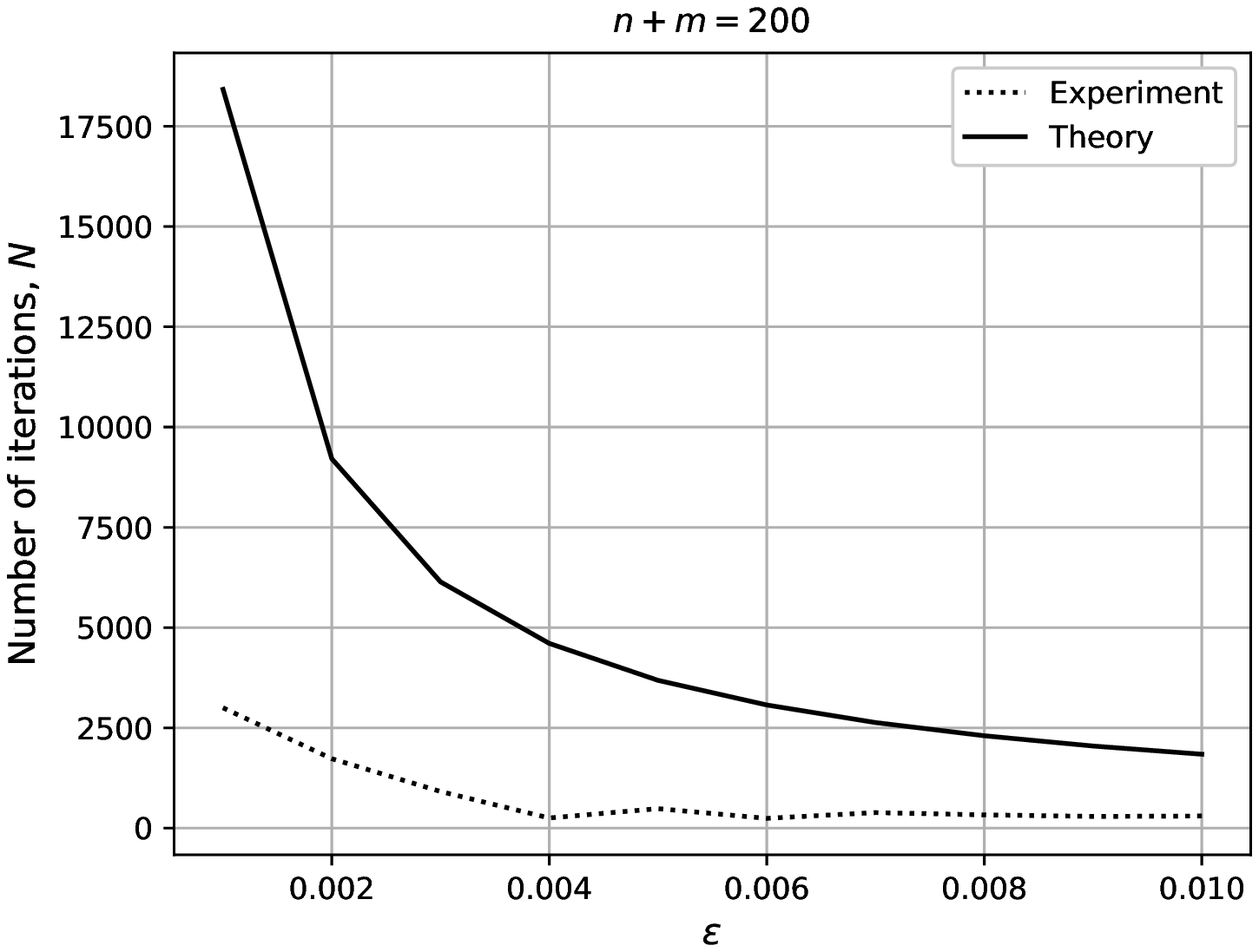}
  \caption{}
  \label{fig:sub2}
\end{subfigure}
\caption{The first kind of experiments with matrix games. Dependence between
experimental and theoretical number of iterations
for different~$\varepsilon$. (a) $10 \, \times \, 10$ matrix $A$ 
from \eqref{eq_sad_prob}, (b) $100 \, \times \, 100$ matrix $A$ from 
\eqref{eq_sad_prob}.}
\label{fig_exp10_m1}
\end{figure}

First of all, we calculate experimental numbers of iterations
for adaptive proximal
method for VI~\cite{UMP_MAIN} and
compare these numbers with theoretical estimation $N=\frac{C}{\varepsilon}$, for some $C>0$~\cite{Nemirovski_2004}.
For that kind of experiments 
we run simulations on two classes of random matrix games: $10 \times 10$ and $100 \times 100$ normally distributed payoff
matrices. For the first setting we create 50 games at random and
calculate average experimental number of iterations over all games.  
Fig.~\ref{fig_exp10_m1} shows the results for different $\varepsilon$.
The experimental results are better than 
theoretical estimation in all cases.

\begin{figure}
\centering
\begin{subfigure}{.5\textwidth}
  \centering
  \includegraphics[width=1.1\linewidth]{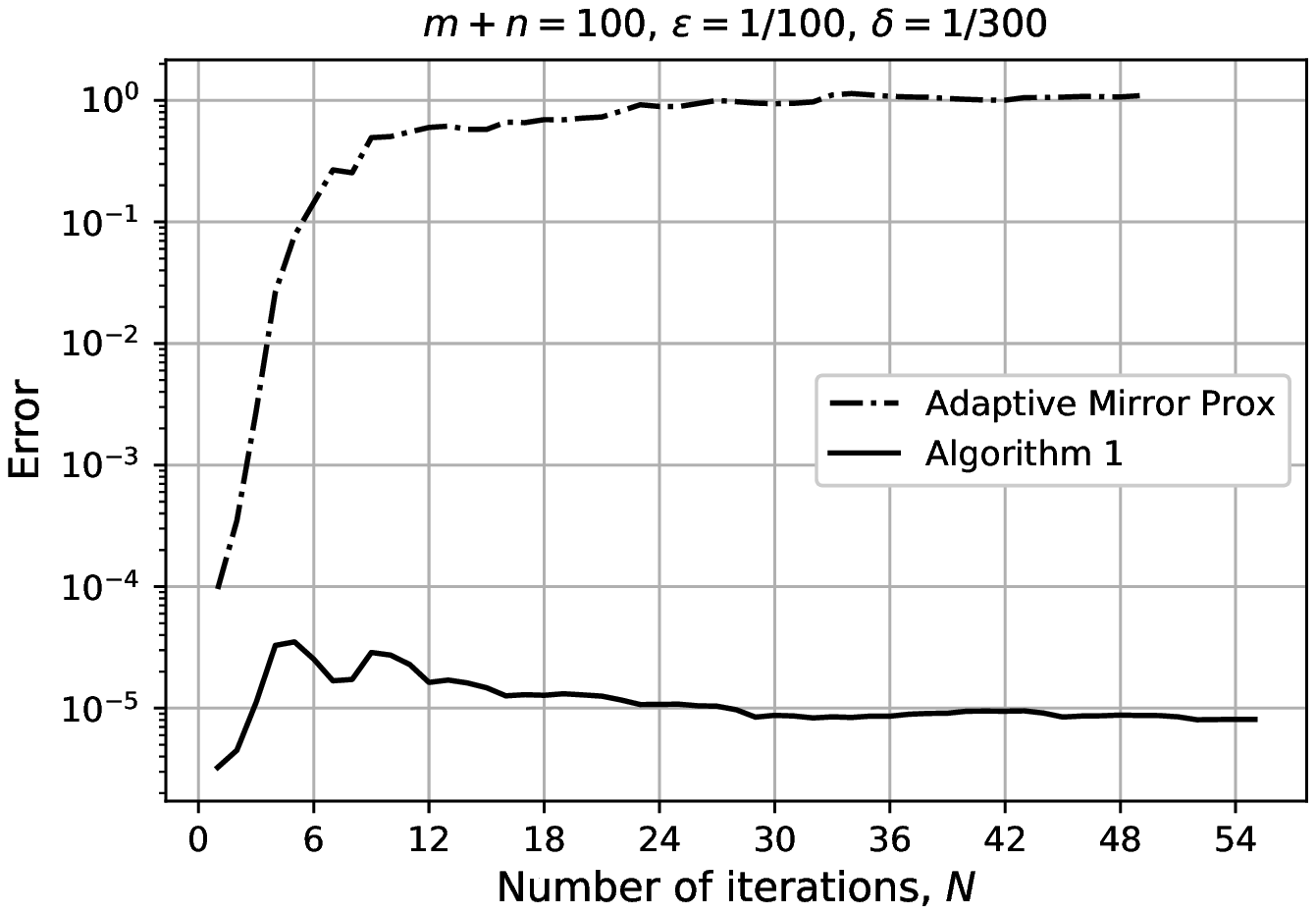}
  \caption{}
  \label{fig:sub3}
\end{subfigure}%
\begin{subfigure}{.5\textwidth}
  \centering
  \includegraphics[width=1.1\linewidth]{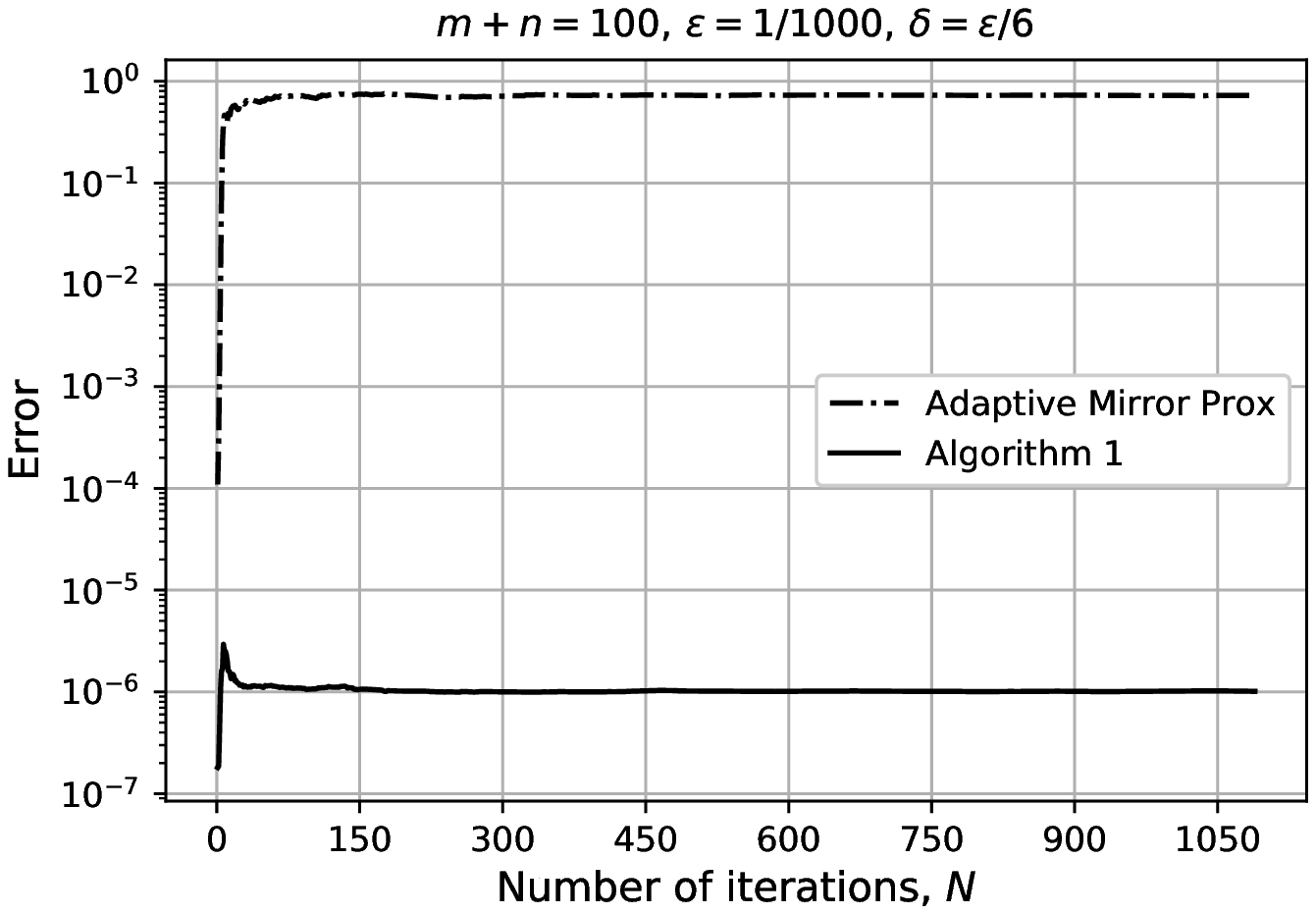}
  \caption{}
  \label{fig:sub4}
\end{subfigure}
\caption{The second kind of experiments with matrix games. The logarithmic scale on the Error-axis. (a) $\varepsilon = 1/100, \, \delta = 1/300$, (b) $\varepsilon = 1/1000, \, \delta = 1/6000$.}
\label{fig_2exp100-1}
\end{figure}

\begin{figure}
\centering
\includegraphics[width=0.8\textwidth]{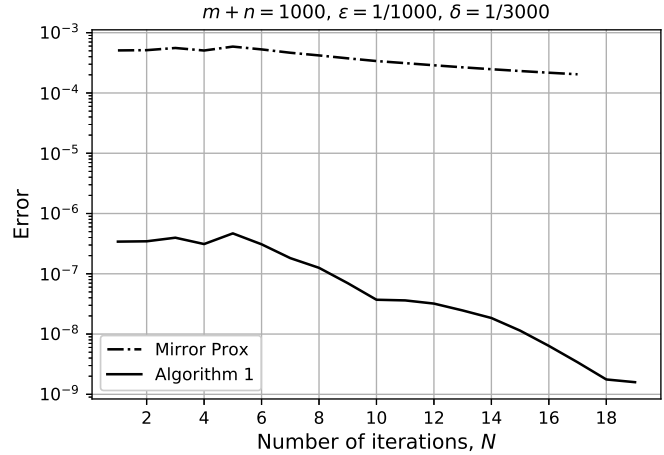}
\caption{The second kind of experiments with huge-scale matrix games. The logarithmic scale on the Error-axis. $1000 \, \times \, 1000$ matrix $A$ from 
\eqref{eq_sad_prob}.} \label{fig_2exp1000-1}
\end{figure}

In the second part of experiments with matrix games we compare the proposed Algorithm MPAI with  adaptive Mirror Prox method from~\cite{UMP_MAIN}. In this part of experiments we modified problem~\eqref{eq_sad_prob} by adding 
inexactness (a bounded by~$\delta$) random noise) to the operator~$g$ of VI~\eqref{ad_meth_eq2}. Fig.~\ref{fig_2exp100-1}-\ref{fig_2exp1000-1} show the results of these experiments.

For comparison we calculate the specific values that determine the degree of influence of the inexactness on the final estimate of the decisions' accuracy.
So, not the whole error estimations are compared, but only improved by our approach part of the
error estimations.

For Algorithm~\ref{algorithm:new_GMP}
this specific value is equal to 
(see~\eqref{ineq:estimateofGMP}):
\begin{equation}
\label{eq_sec5_er_alg1}
\frac{1}{S_N}\sum\limits_{k=0}^{N-1}\frac{\delta^{k+1}}{L^{k+1}}
\left \|y^{k+1}-x^{k+1} \right \|
\end{equation}
and for adaptive Mirror Prox \cite{UMP_MAIN} method we can estimate the analogous value in the following way:
    $$
    \frac{1}{S_N}\sum\limits_{k=0}^{N-1} \frac{\delta}{L^{k+1}} 
    \left \|y^{k+1}-x^{k+1}
    \right \|.
    $$
According to the scheme of the proofs of Theorems \ref{ad_meth_th1} and \ref{ad_meth_th2}, we can obtain the analogous value for non-adaptive Mirror
Prox with constant step $1/L$:
\begin{equation}
\label{eq_sec5_er_MP}
    \frac{1}{N}\sum\limits_{k=0}^{N-1} \delta \left \|y^{k+1} - x^{k+1} \right \|.
\end{equation}
It is worth mentioning that estimate~\eqref{eq_sec5_er_alg1} 
should be less
than \eqref{eq_sec5_er_MP}
because of 
adaptive reduction $\delta^{k+1} < \delta$.

We may see that the accumulated error for  proposed MPAI method is smaller than the error for adaptive Mirror Prox method from ~\cite{UMP_MAIN}.

\section{Conclusion}

The paper proposes an analogue of the Mirror Prox method for variational inequalities with adaptive tuning not only for the constant $L_{\nu}$, but also for the magnitude of the operator’s error. Moreover, such an error can set the degree of discontinuity of the operator. It is proved that the proposed method converges with the optimal estimate of complexity $O\left(\frac{1}{\varepsilon}\right)$ for the Lipschitz-continuous operator and the magnitude of the error is limited. It is important that the result applies to a certain class of variational inequalities with bounded (generally, discontinuous) operators ($\nu = 0$). The paper also presents the results of experiments that demonstrate the ability to work with the estimate of the complexity close to $O\left(\ln\left(\frac{1}{\varepsilon}\right)\right)$ even for a problem with bounded  operator ($\nu = 0$) and the experimental comparison between adaptive Mirror Prox and the proposed algorithm for a special application (matrix games with inexactness). We also note that all result of this paper is applicable for VI with relatively smooth operators (with the exception of Remark \ref{RemNonSmooth}). More precisely, the prox-function and Bregman divergence in \eqref{eqv_gen_smooth} may not be strongly convex (for convex optimization problems this situation was studied e.g. in \cite{relatively_Nesterov_2017,inexact_model_Fedor_2019}).

\end{document}